\documentclass[11pt,a4paper]{amsart} 
\usepackage{amssymb,bbm,bm,mathrsfs,enumerate}
\usepackage{epsfig,color}
\usepackage[T1]{fontenc}
\usepackage{hyperref}

\newtheorem{theorem}{Theorem}[section] 
\newtheorem{lemma}[theorem]{Lemma}
\newtheorem{proposition}[theorem]{Proposition}

\newtheorem{corollary}[theorem]{Corollary}

\def\1{{\mathbbm 1}}
\def\build#1_#2^#3{\mathrel{\mathop{\kern 0pt#1}\limits_{#2}^{#3}}}

\def\epsilon{{\varepsilon}}
\def\phi{{\varphi}}

\def\Z{{\mathbb Z}}

\def\S{{\mathfrak S}}

\def\T{{\mathsf T}}
\def\lint{{[\![}}
\def\rint{{]\!]}}
\def\int#1#2{{\lint #1,#2\rint}}
\def\id{{\rm id}}

\title{Prefixes of minimal factorisations of a cycle}
\author{Thierry L\'evy}
\address{Thierry L\'evy \newline \indent Universit\'e Pierre et Marie Curie - Laboratoire de Probabilit\'es et Mod\`eles Al\'eatoires \newline \indent  Case courrier 188 - 4, place Jussieu - F-75252 Paris Cedex 05 \newline \indent {\rm \texttt{http://www.proba.jussieu.fr/pageperso/levy/index.html}}}
\date{\today}

\begin{document}
\maketitle

\begin{abstract} We give a bijective proof of the fact that the number of $k$-prefixes of minimal factorisations of the $n$-cycle $(1\ldots n)$ as a product of $n-1$ transpositions is $n^{k-1}\binom{n}{k+1}$. Rather than a bijection, we construct a surjection with fibres of constant size. This surjection is inspired by a bijection exhibited by Stanley between minimal factorisations of an $n$-cycle and parking functions, and by a counting argument for parking functions due to Pollak. 
\end{abstract}

\section{Introduction}

It is very well known that the $n$-cycle $(1\ldots n)$ cannot be written as a product of less than $n-1$ transpositions, and that there are $n^{n-2}$ ways of writing it as a product of exactly $n-1$ transpositions. Among other proofs, the one given by R. Stanley in \cite{Stanley} relies on a bijection between minimal factorisations of $(1\ldots n)$ and parking functions of length $n-1$. The bijection is straightforward in one direction, from factorisations to parking functions, and more complicated in the other, but parking functions are very easily counted, thanks to a cunning argument which Stanley attributes to Pollak.

In \cite{Biane}, P. Biane generalised this result and proved that if $a_{1},\ldots,a_{k}$ are integers at least equal to $2$ such that $(a_{1}-1)+\ldots+(a_{k}-1)=n-1$, then there are $n^{k-1}$ ways of writing the $n$-cycle $(1\ldots n)$ as a product $c_{1}\ldots c_{k}$ where $c_{i}$ is an $a_{i}$-cycle for all $i\in \int{1}{k}$. 


In this paper, we generalise the result in another direction by counting the initial segments of length $k\in \int{0}{n-1}$ of minimal factorisations of $(1\ldots n)$ by transpositions (see \eqref{def Sigma} for a precise definition). The number of these prefixes appears in the computation of the repartition of the eigenvalues of a large random unitary matrix taken under the heat kernel measure (see \cite{LevyAIM}). Using the deep relations between the unitary groups and the symmetric groups, it is possible to make these number appear under their combinatorial definition in this computation, and it is then crucial to be able to determine their value. This was done in \cite{LevyAIM}, where it was proved that the number of these segments is $n^{k-1}\binom{n}{k+1}$. However, the proof given there was rather obscure and the goal of the present paper is to give a bijective proof of this identity. 

The present proof consists in constructing a surjective mapping from the set $\int{1}{n}^{k} \times \binom{\int{1}{n}}{k+1}$ to the set of $k$-prefixes of minimal factorisations, with the property that the fibres of this surjection are exactly the orbits of the shift modulo $n$.

The paper is organised as follows. In Section \ref{postintro}, we describe the set which we want to enumerate and recall some classical facts about the geometry of the Cayley graph of the symmetric group. As a guide and motivation, we also give an informal description of the surjection. In Section \ref{path properties}, we collect various elementary properties of the prefixes of minimal factorisations of an $n$-cycle, in particular those for which the sequence of the smallest terms displaced by each successive transposition is non-decreasing. In Section \ref{permutation}, we describe an action of the symmetric group of order $k$ on the set $k$-prefixes of minimal factorisations which plays a crucial role in the construction of the surjection. This construction is finally presented in Section \ref{surjection}, together with the study of the surjection and the proof of our counting result.

\section{The Cayley graph of the symmetric group}\label{postintro}

The beginning of this section is meant to set up the notation and describe the problem. To start with, given two integers $k$ and $l$ such that $k<l$, we denote by $\int{k}{l}$ the set of integers $\{k,k+1,\ldots,l\}$.

Let $n\geq 1$ be an integer. Let $\S_{n}$ be the symmetric group of order $n$. Let $\T_{n}\subset \S_{n}$ be the subset which consists of all transpositions. It is a conjugacy class of $\S_{n}$ and the Cayley graph of the pair $(\S_{n},\T_{n})$ is defined without ambiguity regarding the order of multiplications. In this note, we endow $\S_{n}$ with the graph distance of this Cayley graph. 

This distance can be computed easily by counting the number of cycles of permutations. For all $\sigma\in \S_{n}$, we denote by $\ell(\sigma)$ the number of cycles of $\sigma$, including the trivial cycles. For example, $\sigma$ is a transposition if and only if $\ell(\sigma)=n-1$. The distance between two permutations $\sigma_{1}$ and $\sigma_{2}$ is simply $n-\ell(\sigma_{1}\sigma_{2}^{-1})$. We denote by $|\sigma|$ the number $n-\ell(\sigma)$. Note that for all permutation $\sigma$, one has $|\sigma^{-1}|=|\sigma|$.

The distance on $\S_{n}$ allows one to define a partial order on $\S_{n}$, by setting $\sigma_{1}\preccurlyeq \sigma_{2}$ if and only if $|\sigma_{2}|=|\sigma_{1}|+|\sigma_{1}^{-1}\sigma_{2}|$. We are interested in computing the number of elements of the following set, defined for all $k\geq 0$ : 
\begin{equation}\label{def Sigma}
\Sigma_{n}(k)=\left\{(\tau_{1},\ldots,\tau_{k})\in (\T_{n})^{k} : |\tau_{1}\ldots \tau_{k}|=k,\; \tau_{1}\ldots \tau_{k}\preccurlyeq (1\ldots n)\right\}.
\end{equation}
We will see the elements $\Sigma_{n}(k)$ as paths in the symmetric group, according to the following convention : if $\gamma=(\tau_{1},\ldots,\tau_{k})$ is an element of $\Sigma_{n}(k)$, we denote for all $l\in \int{0}{k}$ by $\gamma_{l}$ the permutation $\tau_{1}\ldots\tau_{l}$. In particular, $\gamma_{0}$ is the identity.

The condition $|\tau_{1}\ldots\tau_{k}|=k$ in the definition of $\Sigma_{n}(k)$ means that for each $l\in \int{1}{k}$, the multiplication on the right by $\tau_{l}$ reduces by $1$ the number of cycles of the permutation $\tau_{1}\ldots \tau_{l-1}$. This is equivalent to saying that the two elements of $\int{1}{n}$ which are exchanged by $\tau_{l}$ belong to distinct cycles of $\tau_{1}\ldots \tau_{l-1}$.

The condition $\tau_{1}\ldots\tau_{k}\preccurlyeq (1\ldots n)$ means, according to the definition of the partial order, that the chain $(\tau_{1},\ldots,\tau_{k})$ of transpositions can be completed into a minimal factorisation of $(1\ldots n)$, that is, a chain $(\tau_{1},\ldots,\tau_{n-1})\in (\T_{n})^{n-1}$ such that $\tau_{1}\ldots \tau_{n-1}=(1\ldots n)$, or yet in other words, a shortest path from the identity to $(1\ldots n)$.

From this description, it follows that $\Sigma_{n}(k)$ is 
\begin{itemize}
\item empty if $k\geq n$,
\item the set of minimal factorisations of $(1\ldots n)$ if $k=n-1$,
\item the projection of $\Sigma_{n}(n-1)$ on the first $k$ coordinates of $(\T_{n})^{n-1}$ if $k\leq n-1$.
\end{itemize}
In particular, if $(\tau_{1},\ldots,\tau_{k})$ belongs to $\Sigma_{n}(k)$, then $\tau_{1}\ldots \tau_{l}\preccurlyeq (1\ldots n)$ for all $l\in \int{0}{k}$. The following classical lemma enables us to decide when a permutation $\sigma$ satisfies $\sigma \preccurlyeq (1\ldots n)$.

\begin{lemma} Let $\sigma\in \S_{n}$ be a permutation. The relation $\sigma\preccurlyeq (1\ldots n)$ holds if and only if the following two conditions hold:
\begin{enumerate}[1. ]
\item Each cycle of $\sigma$ has the cyclic order induced by $(1\ldots n)$.
\item The partition of $\{1,\ldots,n\}$ by the cycles of $\sigma$ is non-crossing with respect to the cyclic order defined by $(1\ldots n)$.
\end{enumerate}
\end{lemma}

The first condition is equivalent to the following: each cycle of $\sigma$ can be written $(i_{1}\ldots i_{r})$ with $i_{1}<\ldots <i_{r}$. The second condition means that there exist no subset $\{i,j,k,l\}$ of $\int{1}{n}$ with $i<j<k<l$ such that $i$ and $k$ belong to a cycle of $\sigma$ and $j$ and $l$ belong to another cycle of $\sigma$.

It is well known that $\Sigma_{n}(n-1)$ has $n^{n-2}$ elements. On the other extreme, $\Sigma_{n}(0)$ consists in the empty path and $\Sigma_{n}(1)=\T_{n}$ has $\binom{n}{2}$ elements. Our main goal is to give a bijective proof of the equality
\begin{equation}\label{main card}
|\Sigma_{n}(k)|=n^{k-1} \binom{n}{k+1} ,
\end{equation}
by the means of a surjective mapping
\[\int{1}{n}^{k} \times \binom{\int{1}{n}}{k+1}  \to \Sigma_{n}(k),\]
such that the preimage of each element of $\Sigma_{n}(k)$ consists in $n$ elements. 
In order to construct and study this mapping, we will need to get fairly concretely into the structure of the elements of $\Sigma_{n}(k)$ and this is what we begin in the next section. Before this, let us describe informally the surjection. 

Let us start with a sequence $(a_{1},\ldots,a_{k})\in \int{1}{n}^{k}$ and a subset $\{b_{1},\ldots,b_{k+1}\} \subset \int{1}{n}$. Let us reorder $(a_{1},\ldots,a_{k})$ into a non-decreasing sequence $(i_{1}\leq \ldots \leq i_{k})$. Consider a circular bike shed with $n$ spaces labelled from $1$ to $n$ counterclockwise in the natural order, and in which only the spaces labelled $\{b_{1},\ldots,b_{k+1}\}$ are open. A first cyclist enters the shed just after the space $i_{k}$, expores the shed counterclockwise, thus starting from space $i_{k}+1$, and parks into the first open and available space. We denote this space by $j_{k}$. Then, $k-1$ other cyclists park one after the other, starting respectively just after the spaces $i_{k-1},\ldots,i_{1}$. We record the spaces which they occupy as $j_{k-1},\ldots,j_{1}$. At the end of the process, there is exactly one space left vacant among the $k+1$  open ones. If this space is not labelled by $1$, we consider that the procedure has failed and we redo it from the beginning after applying to $(a_{1},\ldots,a_{k})$ and $\{b_{1},\ldots,b_{k+1}\}$ the unique shift modulo $n$ which ensures that the second attempt will not fail. We assume now that our initial data is such that the procedure does not fail.

Since the space $1$ has not been occupied, no cyclist has gone past it in the process and the inequalities $i_{1}<j_{1},\ldots,i_{k}<j_{k}$ hold. Moreover, we shall prove that $((i_{1}\, j_{1}),\ldots,(i_{k}\, j_{k}))$ belongs to $\Sigma_{n}(k)$ (see Lemma \ref{sig}).

Now, let $\sigma\in \S_{k}$ be a permutation such that $(a_{1},\ldots,a_{k})=(i_{\sigma(1)},\ldots,i_{\sigma(k)})$. Let us emphasize that if the first attempt of our parking procedure failed, the sequences which we are considering here are those which we obtained after the shift. This permutation is not unique in general, but we shall prove that the result of the construction is independent of our choice (see Proposition \ref{action}). We want to let $\sigma$ act on the path $((i_{1}\, j_{1}),\ldots,(i_{k}\, j_{k}))$. For this, write $\sigma$ as a product of transpositions of the form $(l \, l+1)$ with $l\in \int{1}{k-1}$ and let these transpositions act successively on $((i_{1}\, j_{1}),\ldots,(i_{k}\, j_{k}))$ as follows: if $i_{l}=i_{l+1}$, do nothing, but if $i_{l}\neq i_{l+1}$, exchange $(i_{l}\, j_{l})$ and $(i_{l+1}\, j_{l+1})$ and conjugate the one with the smallest $i$ by the other.

Let us illustrate this on an example. Take $n=8$, $k=4$, consider the sequence $(1,3,7,1)$ and the subset $\{1,3,5,6,7\}$. The bikes enter the shed just after the spaces $7,3,1,1$ in this order and park respectively in the spaces $1,5,3,6$ (see Figure \ref{shed} below). 
\begin{figure}[h!]
\begin{center}
\includegraphics[width=13cm]{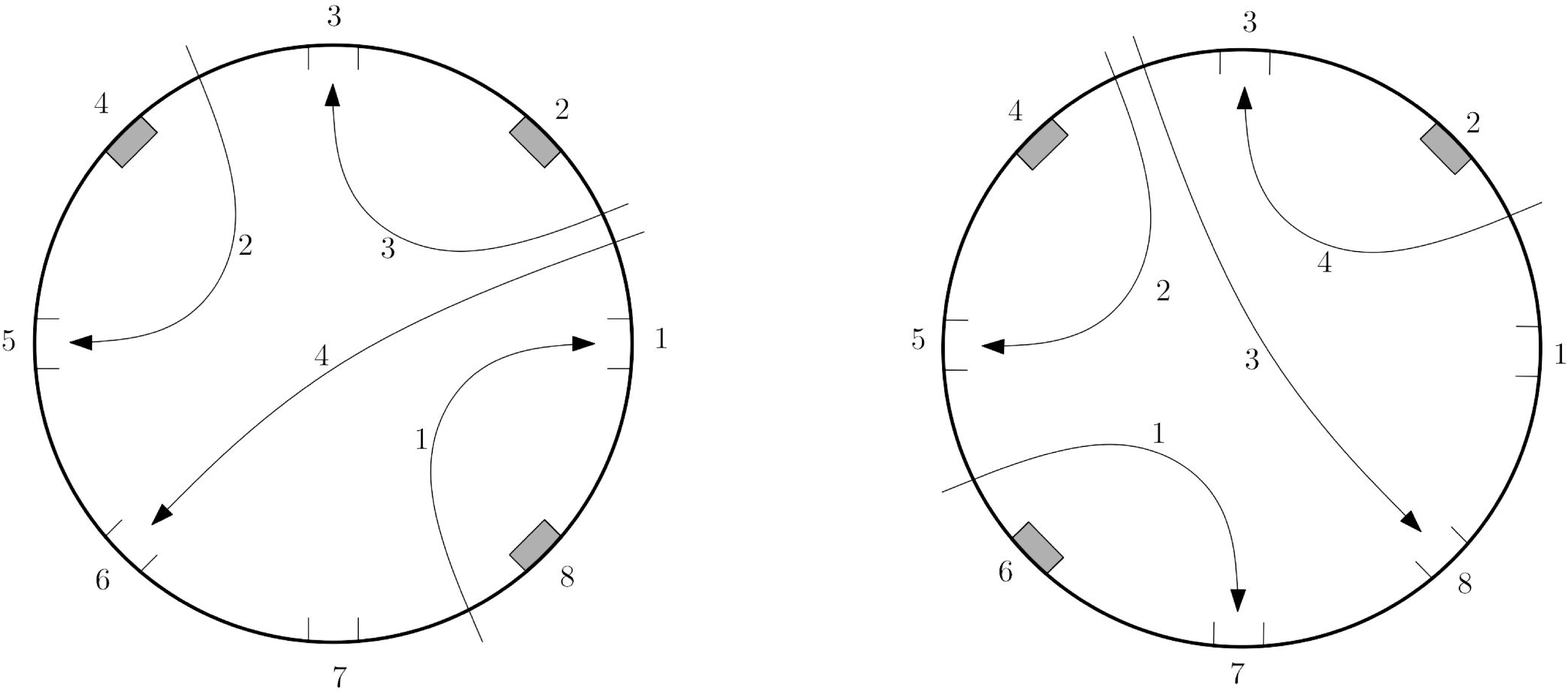}
\caption{\label{shed}\small  The paths of the bikes are labelled by their order of entrance in the shed. On the left-hand side, the original situation. On the right-hand side, the shifted one so that the space left vacant is $1$. Observe that the order of entrance has been modified by the shift.}
\end{center}
\end{figure}

The procedure fails : the empty space is labelled $7$. We must shift everything by $2$ modulo $8$ and redo the parking. The new sequence is $(3,5,1,3)$, the new subset $\{1,3,5,7,8\}$. The bikes enter the shed after the spaces $(5,3,3,1)$ and park at $(7,5,8,3)$. We obtain the chain $((1\, 3),(3\, 8),(3\, 5),(5\, 7))$. A permutation which transforms $(3,5,1,3)$ into $(1,3,3,5)$ is $(1\, 3)(2\, 4)=(2\, 3)(1\, 2)(3\, 4)(2\, 3)$.
The transposition $(2\, 3)$ does not change the chain, then $(3\, 4)$ changes it to $((1\, 3),(3\, 8),(5\, 7),(3\, 7))$, then $(1\, 2)$ to $((3\, 8),(1\, 8),(5\, 7),(3\, 7))$ and finally $(2\, 3)$ to $((3\, 8),(5\, 7),(1\, 8),(3\, 7))$. This is the element of $\Sigma_{8}(4)$ which the surjection produces. It is indeed an element of $\Sigma_{8}(4)$, since $(3\, 8)(5\, 7)(1\, 8)(3\, 7)=(13578)\preccurlyeq (1\ldots 8)$ and $|(13578)|=4$.

\section{Non-decreasing geodesic paths}\label{path properties}

Let us agree on the convention that every time we write a transposition under the form $(i\, j)$, we mean $i<j$.


For all permutation $\pi\in \S_{n}$ and all $x\in \int{1}{n}$, we denote by $C_{\pi}(x)$ the cycle of $\pi$ which contains $x$. We will sometimes forget the cyclic order on $C_{\pi}(x)$ and consider it merely as a subset of $\int{1}{n}$. The following result is largely inspired by the proof of Theorem 3.1 in the work \cite{Stanley} of R. Stanley.

\begin{lemma} \label{prop mini}
Let $\gamma=((i_{1}\, j_{1}),\ldots,(i_{k}\; j_{k}))$ be an element of $\Sigma_{n}(k)$. Choose $l\in \int{1}{k}$. The following properties hold.
\begin{enumerate}[1.]
\item $i_{l}<C_{\gamma_{l-1}}(j_{l})$ and $i_{l}$ is the largest element of $C_{\gamma_{l-1}}(i_{l})$ with this property.
\item  $j_{l}=\max C_{\gamma_{l-1}}(j_{l})$.
\item  $i_{l}<C_{\gamma_{l-1}}(i_{l}+1)$.
\item  If $i_{l}+1\notin \{i_{1},i_{2},\ldots,i_{l-1}\}$, then $C_{\gamma_{l-1}}(i_{l}+1)=\{i_{l}+1\}$.
\item  If $k=n-1$, $i_{l}=\max\{i_{1},i_{2},\ldots,i_{n-1}\}$ and $l=\max\{s\in \int{1}{k} : i_{s}=i_{l}\}$, then $j_{l}=i_{l}+1$.
\end{enumerate}
\end{lemma}

\begin{proof} Since $|\gamma_{l}|=|\gamma_{l-1}|+1$,  $i_{l}$ and $j_{l}$ belong to distinct cycles of $\gamma_{l-1}$ and to the same cycle of $\gamma_{l}$. The cycle of $\gamma_{l}$ which contains $i_{l}$ and $j_{l}$ has the cyclic order induced by $(1\ldots n)$, so that it is of the form $(x_{1}<\ldots < x_{r} < i_{l} < y_{1}<\ldots < y_{s}< j_{l} < z_{1}<\ldots <z_{t})$. The cycles of $\gamma_{l-1}$ which contain $i_{l}$ and $j_{l}$ are thus respectively $(x_{1}\ldots x_{r} \;  i_{l} \;  z_{1} \ldots z_{t})$ and $(y_{1}\ldots  y_{s}\;  j_{l})$. This proves the first two assertions.

The second part of the first assertion implies that $i_{l}+1\notin C_{\gamma_{l-1}}(i_{l})$. If $i_{l}+1\in C_{\gamma_{l-1}}(j_{l})$, then third assertion follows from the first. Let us assume that $i_{l}+1\notin C_{\gamma_{l-1}}(j_{l})$. In this case, $C_{\gamma_{l-1}}(i_{l}+1)=C_{\gamma_{l}}(i_{l}+1)$. Suppose that there is an element $x$ in $C_{\gamma_{l-1}}(i_{l}+1)$ such that $x<i_{l}$. Then the quadruplet $x<i_{l}<i_{l}+1<y_{l}$ would violate the non-crossing condition on the cycles of $\gamma_{l}$ imposed by the condition $\gamma_{l}\preccurlyeq (1\ldots n)$. This concludes the proof of the third assertion.

Let us assume that $i_{l}+1\notin \{i_{1},\ldots,i_{l-1}\}$. Let $r$ be the smallest element of $\int{1}{k}$, if it exists, such that the cycle of $i_{l}+1$ in $\gamma_{r}$ is not reduced to the singleton $\{i_{l}+1\}$. We must have $i_{r}=i_{l}+1$ or $j_{r}=i_{l}+1$. If $i_{r}=i_{l}+1$, then our assumption implies $r\geq l$, so that $C_{\gamma_{l-1}}(i_{l}+1)=\{i_{l}+1\}$. If $j_{r}=i_{l}+1$, then $i_{r}\in C_{\gamma_{r}}(i_{l}+1)$. Since $i_{r}\leq i_{l}$ and thanks to the third assertion, this implies that $r\geq l$, so that in this case also we have $C_{\gamma_{l-1}}(i_{l}+1)=\{i_{l}+1\}$. This proves the fourth assertion.

Let us assume that $k=n-1$, $i_{l}=\max\{i_{1},i_{2},\ldots,i_{n-1}\}$ and $l=\max\{s \in \int{1}{k}: i_{s}=i_{l}\}$. We are thus looking, in a minimal factorisation of $(1\ldots n)$, at the last occurrence of the largest $i$. Let, as before, $r$ be the smallest element of $\int{1}{n-1}$ such that $C_{\gamma_{r}}(i_{l}+1)$ is not reduced to the singleton $\{i_{l}+1\}$. Since $(1\ldots n)$ has no fixed point, we know for sure that $r$ exists. By maximality of $i_{l}$, we have $i_{l}+1=j_{r}$. If $r>l$, then by maximality of $i_{l}$ and of $l$, we have $i_{r}<i_{l}$. Thus, the quadruplet $i_{r}<i_{l}<i_{l}+1<j_{l}$ violates the non-crossing condition on the cycles of the permutation $\gamma_{r}$. This proves the fifth assertion.
\end{proof}

We now make an observation of monotonicity.

\begin{lemma}\label{monotone} Consider $\gamma=((i_{1}\, j_{1}),\ldots,(i_{k}\; j_{k})) \in \Sigma_{n}(k)$ and $l,m\in \int{1}{k}$ with $l<m$.
\begin{enumerate}[1.]
\item If $i_{l}=i_{m}$, then $j_{l}>j_{m}$.
\item If $j_{l}=j_{m}$, then $i_{l}>i_{m}$.
\end{enumerate}
\end{lemma}

\begin{proof} Let us assume that $i_{l}=i_{m}$. We have $j_{m}\notin C_{\gamma_{m-1}}(i_{m})$ and $j_{l}\in C_{\gamma_{m-1}}(i_{l})=C_{\gamma_{m-1}}(i_{m})$.  In particular, $j_{l}\neq j_{m}$. Both $i_{m}$ and $j_{l}$ belong to $C_{\gamma_{m-1}}(i_{m})$ but, according to the first assertion of Lemma \ref{prop mini},  $i_{m}$ is the largest element of $C_{\gamma_{m-1}}(i_{m})$ which is smaller than any element of $C_{\gamma_{m-1}}(j_{m})$. Hence, there exists $x\in C_{\gamma_{m-1}}(j_{m})$ such that $i_{m}=i_{l}<x<j_{l}$. The inequality $j_{l}<j_{m}$ cannot hold, for then the quadruplet $i_{l}<x<j_{l}<j_{m}$ would violate the non-crossing property of the cycles of $\gamma_{m-1}$.

The second assertion follows from the first and the existence of a simple involution of $\Sigma_{n}(k)$, which we describe in Lemma \ref{involution} below.
\end{proof}

\begin{lemma} \label{involution} Let $((i_{1}\, j_{1}),\ldots,(i_{k}\; j_{k}))$ be an element of $\Sigma_{n}(k)$. Then the chain of transpositions $((n+1-j_{k}\, n+1-i_{k}),\ldots,(n+1-j_{1}\; n+1-i_{1}))$ is also an element of $\Sigma_{n}(k)$. 
\end{lemma}

\begin{proof} Let $\phi\in \S_{n}$ be the involution which exchanges $i$ and $n+1-i$ for all $i\in \{1,\ldots,n\}$. The point is the identity $\phi (1\ldots n)^{-1}\phi^{-1}=(1\ldots n)$. Let $(\tau_{1},\ldots,\tau_{k})$ be an element of $\Sigma_{n}(k)$. Then on one hand $|\phi \tau_{k}\ldots \tau_{1} \phi^{-1}|=|\tau_{k}\ldots \tau_{1}|=|\tau_{1}\ldots \tau_{k}|=k$. On the other hand, we have the equality 
$|(1\ldots n)^{-1} \phi \tau_{k}\ldots \tau_{1} \phi^{-1}|=|\phi^{-1} (1\ldots n)^{-1} \phi \tau_{k}\ldots \tau_{1}|=|(1\ldots n) (\tau_{1}\ldots \tau_{k})^{-1}|=n-1-k$. Hence, $\phi \tau_{k}\ldots \tau_{1} \phi^{-1} \preccurlyeq (1\ldots n)$. Finally, $(\phi \tau_{k}\phi^{-1},\ldots,\phi\tau_{1}\phi^{-1})$ belongs to $\Sigma_{n}(k)$.
\end{proof}

It turns out that the elements of $\Sigma_{n}(k)$ for which the sequence $(i_{1},\ldots,i_{n})$ is non-decreasing are easy to describe and to characterise. We call them {\em non-decreasing paths} and we denote by $\Sigma_{n}^{*}(k)$ the subset of $\Sigma_{n}(k)$ which they constitute.

\begin{lemma}\label{sm} Let $\gamma=((i_{1}\, j_{1}),\ldots,(i_{k}\; j_{k}))$ be an element of $\Sigma_{n}^{*}(k)$. The following properties hold.
\begin{enumerate}[1.]
\item The sequence $(j_{1},\ldots,j_{k})$ has no repetitions.
\item For all $l\in \int{1}{k}$, $j_{l}$ is a fixed point of $\gamma_{l-1}$.
\item For all $l\in \int{1}{k}$, $\gamma_{l}$ is obtained from $\gamma_{l-1}$ by inserting $j_{l}$ into the cycle of $i_{l}$ immediately after $i_{l}$.
\item For all $m\in \int{1}{k}$, the support of $\gamma_{m}$ is $\bigcup_{l=1}^{m} \{i_{l}\}\cup \{j_{l}\}$.
\end{enumerate}
\end{lemma}

\begin{proof} The second assertion of Lemma \ref{monotone} implies that each repetition in the sequence $(j_{1},\ldots,j_{k})$ corresponds to a descent in the sequence $(i_{1},\ldots,i_{k})$, hence the first assertion. 

For all $l\in \int{1}{k}$, we have $j_{l}>i_{l}\geq \ldots \geq i_{1}$ and, by the first assertion, $j_{l} \notin\{j_{1},\ldots,j_{l-1}\}$, so that $j_{l}$ is a fixed point of $\gamma_{l-1}$. This is the second assertion.

For all $l\in \int{1}{k}$, the second assertion implies that $\gamma_{l}(i_{l})=j_{l}$, and we have $\gamma_{l}(j_{l})=\gamma_{l-1}(i_{l})$. This is exactly the third assertion.

The fourth assertion follows from the second and third assertions by induction on $k$. 
\end{proof}

\begin{proposition} \label{carac} Consider $((i_{1}\, j_{1}),\ldots,(i_{k}\, j_{k}))\in (\T_{n})^{k}$. Assume that $i_{1}\leq \ldots \leq i_{k}$. The following properties are equivalent.
\begin{enumerate}[1.]
\item $((i_{1}\, j_{1}),\ldots,(i_{k}\, j_{k}))\in \Sigma_{n}(k)$.
\item For all $l,m\in \{1,\ldots,n-1\}$ such that $l<m$, one either has $j_{l}\leq i_{m}$ or $j_{l}>j_{m}$. 
\end{enumerate}

\end{proposition}

\begin{proof} Let us prove that the first property implies the second. For this, let us choose $\gamma=((i_{1}\, j_{1}),\ldots,(i_{k}\, j_{k}))\in \Sigma_{n}(k)$ and $l,m$ with $1\leq l<m\leq n-1$. It follows from the first assertion of Lemma \ref{sm} that $j_{l}\neq j_{m}$. Let us assume by contradiction that $i_{m}<j_{l}< j_{m}$. Then, by Lemma \ref{monotone}, $i_{l}<i_{m}$. Hence, $i_{l}<i_{m}<j_{l}<j_{m}$. We know, by the second assertion of Lemma \ref{sm}, that $C_{\gamma_{m-1}}(j_{m})=\{j_{m}\}$. We claim that $i_{m}\notin C_{\gamma_{m-1}}(i_{l})$. Otherwise, since $j_{l}\in C_{\gamma_{m-1}}(i_{l})$, the element $j_{l}$ of $C_{\gamma_{m-1}}(i_{m})$ would satisfy both $j_{l}>i_{m}$ and $j_{l}<C_{\gamma_{m-1}}(j_{m})$, in contradiction with the first assertion of Lemma \ref{prop mini}. 

It follows from this argument that neither $i_{m}$ nor $j_{m}$ belong to the common cycle of $i_{l}$ and $j_{l}$ in $\gamma_{m-1}$. Hence, the two cycles $C_{\gamma_{m}}(i_{l})=C_{\gamma_{m}}(j_{l})$ and $C_{\gamma_{m}}(i_{m})=C_{\gamma_{m}}(j_{m})$ are distinct. Since $i_{l}<i_{m}<j_{l}<j_{m}$, this contradicts the non-crossing property of the cycles of $\gamma_{m}$. \\

Let us now prove that the second property implies the first. To start with, observe that the second property implies 
that $j_{1},\ldots,j_{k}$ are pairwise distinct and that the equality $i_{l}=i_{m}$ for $l<m$ implies $j_{l}>j_{m}$.

We now proceed by induction on $k$. If $k=1$, then the result is true because $\Sigma_{n}(1)=\T_{n}$. Let us assume that the result holds for paths of length up to $k-1$ and let us consider a path $\gamma=((i_{1}\, j_{1}),\ldots,(i_{k}\, j_{k}))\in (\T_{n})^{k}$ such that $i_{1}\leq \ldots \leq i_{k}$ and the second property holds. By induction, $\gamma_{k-1}$ is a product of $n-k+1$ cycles with the cyclic order induced by $(1\ldots n)$ and which form a non-crossing partition of $\{1,\ldots,n\}$.

By the third assertion of Lemma \ref{sm}, $\gamma_{k}$ a product of $n-k$ cycles. 

Let us prove that the cyclic order of the new cycle is the order induced by $(1\ldots n)$.
We certainly have $i_{k}<j_{k}$ and we claim that $i_{k}<j_{k}<\gamma_{k-1}(i_{k})$ in the cyclic order of $(1\ldots n)$, which means exactly that $\gamma_{k-1}(i_{k})\leq i_{k}$ or $\gamma_{k-1}(i_{k})>j_{k}$. But $\gamma_{k-1}(i_{k})$ is either $i_{l}$ for some $l\in \int{1}{k-1}$, in which case $\gamma_{k-1}(i_{k})\leq i_{k}$, or $\gamma_{k-1}(i_{k})$ is $j_{l}$ for some $l\in \int{1}{k-1}$, in which case $\gamma_{k-1}(i_{k})\leq i_{k}$ or $\gamma_{k-1}(i_{k})> j_{k}$, by the main assumption. 

Let us finally prove that the cycles of $\gamma_{k}$ form a non-crossing partition. The only way this could not be true is if some cycle contained two elements $x$ and $y$ such that $i_{k} < x <j_{k}<y < \gamma_{k-1}(i_{k})$ in the cyclic order. But the any $x$ such that $i_{k}<x<j_{k}$ does neither belong to $\{i_{1},i_{2},\ldots,i_{k}\}$ nor to $\{j_{1},j_{2},\ldots,j_{k}\}$ and hence is a fixed point of $\gamma_{k}$. 
\end{proof}

This proposition allows us to prove that a non-decreasing path $\gamma\in \Sigma_{n}^{*}(k)$ is completely determined by the sequence $(i_{1},\ldots,i_{k})$ and the support of $\gamma_{k}$.

\begin{corollary}\label{determine} Let $\gamma=((i_{1}\, j_{1}),\ldots,(i_{k}\; j_{k}))$ be an element of $\Sigma_{n}^{*}(k)$. For all $l\in \int{1}{k}$, $j_{l}$ is the minimum of the intersection of $\int{i_{l}+1}{n}$ with the support of $\gamma_{l}$.

Moreover, if $\tilde\gamma=((i_{1}\, \tilde j_{1}),\ldots,(i_{k}\; \tilde j_{k}))$ is another element of $\Sigma_{n}^{*}(k)$ such that $\tilde \gamma_k$ and $\gamma_{k}$ have the same support, then $\tilde \gamma=\gamma$.
\end{corollary}

\begin{proof} The support of $\gamma_{l}$ is $\bigcup_{s=1}^{l} \{i_{s}\} \cup \{j_{1},\ldots,j_{l}\}$. For all $s<l$, we have $i_{s}\leq i_{l}$ and, by Proposition \ref{carac}, $j_{s}\leq i_{l}$ or $j_{s}>j_{l}$. The first assertion follows.

Let us prove the second assertion by induction on $k$. The result is true for $k=0$. Let us assume that is has been proved for paths of length up to $k-1$. By the first assertion, $\tilde j_{k}=j_{k}$. Hence, $\delta=((i_{1}\, j_{1}),\ldots,(i_{k-1}\; j_{k-1}))$ and $\tilde \delta=((i_{1}\, \tilde j_{1}),\ldots,(i_{k-1}\; \tilde j_{k-1}))$ are two elements of $\Sigma_{n}(k-1)$ such that $\tilde \delta_{k-1}$ and $\delta_{k-1}$ have the same support. By induction, they are equal.
\end{proof}

\section{Permutation of geodesic paths}\label{permutation}

In this section, we will describe an action of the group $\S_{k}$ on $\Sigma_{n}(k)$. More precisely, let us consider the projection $P:\Sigma_{n}(k)\to \int{1}{n-1}^{k}$ which sends the chain $((i_{1}\, j_{1}),\ldots,(i_{k}\, j_{k}))$ to the sequence $(i_{1},\ldots,i_{k})$. The group $\S_{k}$ acts naturally on $\int{1}{n-1}^{k}$ by the formula $\sigma\cdot (i_{1},\ldots,i_{k})=(i_{\sigma^{-1}(1)},\ldots,i_{\sigma^{-1}(k)})$ and we will endow $\Sigma_{n}(k)$ with an action of $\S_{k}$ such that $P$ is an equivariant mapping which preserves the stabilisers. This last condition is equivalent to the fact that the restriction of $P$ to each orbit of $\S_{k}$ in $\Sigma_{n}(k)$ is an injection. 

In order to define the action of $\S_{k}$ on $\Sigma_{n}(k)$, we will use the classical action of the braid group $B_{k}$ on the product of $k$ copies of an arbitrary group $G$ (see for example \cite{Artin}). If $\beta_{1},\ldots,\beta_{k}$ are the usual generators of $B_{k}$, this action is given by the formula
\[\beta_{l}\cdot (g_{1},\ldots,g_{k})=(g_{1},\ldots,g_{l+1},g_{l+1}^{-1}g_{l}g_{l+1},\ldots,g_{k}),\]
valid for all $(g_{1},\ldots,g_{k})\in G^{k}$ and all $l\in \int{1}{k-1}$. Observe that if $T\subset G$ is a conjugacy class, then $T^{k}$ is stable under this action. Moreover, the product map $(g_{1},\ldots,g_{n})\mapsto g_{1}\ldots g_{n}$ is invariant under this action.

Let us denote by $\sigma_{1}=(1\, 2), \ldots,\sigma_{k-1}=(k-1\, k)$ the Coxeter generators  of $\S_{k}$, so that the natural mophism $B_{k}\to \S_{k}$ sends $\beta_{l}$ to $\sigma_{l}$ for all $l\in \int{1}{k-1}$. Consider $\gamma=((i_{1}\, j_{1}),\ldots,(i_{k}\, j_{k}))$ in $\Sigma_{n}(k)$ and $l\in \int{1}{k-1}$. Set
\begin{equation}\label{def action}
\sigma_{l}\cdot \gamma=\left\{\begin{array}{cc} \gamma & \mbox{if } i_{l}=i_{l+1}, \\
\beta_{l}\cdot \gamma & \mbox{if } i_{l}<i_{l+1}, \\
\beta_{l}^{-1}\cdot \gamma & \mbox{if } i_{l}>i_{l+1}.\end{array}\right.
\end{equation}
Since the action of the braid group preserves the ordered product of the components, $\sigma_{l}\cdot \gamma$ belongs to $\Sigma_{n}(k)$.

Practically, $\sigma_{l}\cdot \gamma$ is obtained from $\gamma$ by doing nothing if $i_{l}=i_{l+1}$, and otherwise, by swapping the $l$-th and $(l+1)$-th elements of $\gamma$ and conjugating the one with the smallest $i$ by the other. In this way, the transposition with the largest $i$ is not modified, and only the $j$ of the other is affected. For example, if $k=2$,
\begin{align*}
\sigma_{1}\cdot((1\, 3),(1\, 2))&=((1\, 3),(1\, 2)),\\
\sigma_{1}\cdot((1\, 2),(2\, 3))&=((2\, 3),(1\, 3)),\\
\sigma_{1}\cdot((2\, 3),(1\, 3))&=((1\, 2),(2\, 3)).
\end{align*}
A straightforward inspection will convince the reader of the following fact.

\begin{lemma} \label{equiv} For all $\gamma \in \Sigma_{n}(k)$ and $l\in \int{1}{k-1}$, one has $P(\sigma_{l}\cdot \gamma)=\sigma_{l}\cdot P(\gamma)$. 

Moreover, if $\gamma=((i_{1}\, j_{1}),\ldots,(i_{k}\, j_{k}))$ and $\sigma_{l}\cdot \gamma=((i_{\sigma_{l}^{-1}(1)}\, \tilde j_{1}),\ldots,(i_{\sigma_{l}^{-1}(k)}\, \tilde j_{k}))$, then the sets $\bigcup_{l=1}^{k}\{i_{l}\}\cup \{j_{l}\}$ and $\bigcup_{l=1}^{k}\{i_{l}\}\cup \{\tilde j_{l}\}$ are equal.
\end{lemma}

We will show at the end of this section that the set $\bigcup_{l=1}^{k}\{i_{l}\}\cup \{j_{l}\}$ is the support of $\gamma_{k}$. For the time being, let us prove that \eqref{def action} defines an action of $\S_{k}$ on $\Sigma_{n}(k)$.

\begin{proposition}\label{action} The action of the Coxeter generators of $\S_{k}$ on $\Sigma_{n}(k)$ defined by \eqref{def action} extends to an action of $\S_{k}$. 

Moreover, the mapping $P:\Sigma_{n}(k)\to \int{1}{n-1}^{k}$ is equivariant and preserves the stabilisers : for all $\gamma\in \Sigma_{n}(k)$ and all $\pi\in \S_{k}$, one has $\pi\cdot P(\gamma)=P(\gamma)$ if and only if $\pi\cdot \gamma=\gamma$.
\end{proposition}

\begin{proof} We must prove that the operations which we have defined satisfy the Coxeter relations $\sigma_{l}^{2}=\id$ for $l\in \int{1}{k-1}$, $(\sigma_{l}\sigma_{m})^{2}=\id$ for $l,m\in \int{1}{k-1}$ with $|l-m|\geq 2$, and $(\sigma_{l}\sigma_{l+1})^{3}=\id$ for $l\in \int{1}{n-2}$.

The first relation follows from Lemma \ref{equiv}. Indeed, $\sigma_{l}\cdot (\sigma_{l}\cdot \gamma)$ is either $\gamma$ or $\beta_{l}\beta_{l}^{-1}\cdot \gamma$ or $\beta_{l}^{-1}\beta_{l}\cdot \gamma$, hence in any case $\gamma$. The second relation is equivalent to $\sigma_{l}\cdot(\sigma_{m}\cdot \gamma)=\sigma_{m}\cdot(\sigma_{l}\cdot \gamma)$ and it clearly holds for $|l-m|\geq 2$. In order to prove the third relation, there are six cases to consider, correponding to the possible relative positions of $i_{l}$, $i_{l+1}$ and $i_{l+2}$. In each case, the relation $\beta_{l}\beta_{l+1}\beta_{l}=\beta_{l+1}\beta_{l}\beta_{l+1}$ implies the relation $(\sigma_{l}\sigma_{l+1})^{3}=\id$.

We have thus an action of the symmetric group $\S_{k}$ on $\Sigma_{n}(k)$. By Lemma \ref{equiv}, the mapping $P$ is equivariant under this action and the natural action on $\int{1}{n-1}^{k}$. If $\gamma\in \Sigma_{n}(k)$ and $\pi\in \S_{k}$ satisfy $\pi\cdot \gamma=\gamma$, then $\pi\cdot P(\gamma)=P(\pi\cdot\gamma)=P(\gamma)$. Finally, let us prove that $\pi\cdot P(\gamma)=P(\gamma)$ implies $\pi\cdot \gamma=\gamma$. Let us choose $\gamma\in \Sigma_{n}(k)$. A permutation $\pi$ stabilises $P(\gamma)$ if and only if its cycles are contained in the level sets of the mapping $1\mapsto i_{1},\ldots,k\mapsto i_{k}$. Thus, the stabiliser of $P(\gamma)$ is generated by the transpositions which it contains, and we may restrict ourselves to the case where $\pi$ is a transposition $(l\, m)$ with $i_{l}=i_{m}$. We have $(l\, m)=\sigma_{l}\ldots \sigma_{m-2}\sigma_{m-1}\sigma_{m-2}\ldots \sigma_{l}$ and $\sigma_{m-2}\ldots \sigma_{l}=(m-1\ldots l)$. Since $P$ is equivariant, the transpositions which are at the positions $m-1$ and $m$ in the chain $\sigma_{m-2}\ldots \sigma_{l}\cdot \gamma$ have respectively $i_{l}$ and $i_{m}$ as their smallest element. Since $i_{l}=i_{m}$, we find 
\[(l\, m)\cdot \gamma=\sigma_{l}\ldots \sigma_{m-2}\sigma_{m-1}\sigma_{m-2}\ldots \sigma_{l}\cdot \gamma=\sigma_{l}\ldots \sigma_{m-2}\sigma_{m-2}\ldots \sigma_{l}\cdot \gamma=\gamma,\]
as expected.
\end{proof}

\begin{corollary}\label{support} Let $\gamma=((i_{1}\, j_{1}),\ldots,(i_{k}\, j_{k}))$ be an element of $\Sigma_{n}(k)$. The support of $\gamma_{k}=(i_{1}\, j_{1})\ldots (i_{k}\, j_{k})$ is the set $\bigcup_{l=1}^{k}\{i_{l}\}\cup \{j_{l}\}$.
\end{corollary}

\begin{proof} The action of $\S_{k}$ on $\Sigma_{n}(k)$ preserves both the support of $\gamma_{k}$ and the set to which we wish to show that it is equal. Since every orbit contains a non-decreasing chain, that is, a chain for which the sequence $(i_{1},\ldots,i_{n})$ is non-decreasing, we may assume that the element $\gamma$ which we are considering has this property, and apply the fourth assertion of Lemma \ref{sm}.
\end{proof}

In the context of minimal factorisations of a cycle, the natural action of the braid group is called the Hurwitz action and it is known to be transitive (see for example \cite{Ripoll}). The action which we have defined here is germane to this action but different, as it is an action of the symmetric group. In \cite{Biane}, P. Biane defined yet another similar action of the symmetric group on minimal factorisations of a cycle as a product of cycles. The proof of Lemma \ref{equiv} is inspired by this work.

\section{The main surjection}\label{surjection}

We have now gathered the information necessary to define the surjection which is our main goal. Although we do not develop this point, our construction is inspired by the enumeration of parking functions by an argument due to Pollak, and the bijection constructed by Stanley between parking functions and minimal factorisations of an $n$-cycle (see \cite{Stanley}).

Let us start by formalising the parking process in a bike shed described in Section \ref{postintro}. Given a sequence $E=(e_{1},\ldots,e_{k})\in \int{1}{n}^{k}$ of entry points and a set $O=\{o_{1},\ldots,o_{k+1}\}\subset \int{1}{n}$ of open spaces, we define a sequence of parking spaces $(p_{1},\ldots,p_{k})$ by backwards induction, by setting
\begin{equation}\label{def jk}
p_{k}=(1\ldots n)^{r} e_{k}, \mbox{ where } r=\min\left\{s\in \int{1}{n} : (1\ldots n)^{s} e_{k} \in \{o_{1},\ldots,o_{k+1}\}\right\}
\end{equation}
and, assuming that $p_{k},\ldots,p_{l+1}$ have been defined,
\begin{equation}\label{def other j}
p_{l}=(1\ldots n)^{r} e_{l}, \mbox{ where } r=\min\{s\in \int{1}{n} : (1\ldots n)^{s} e_{l} \in \{o_{1},\ldots,o_{k+1}\}\setminus\{p_{l+1},\ldots,p_{k}\}\}.
\end{equation}
We call this construction the parking process and write $\Pi(E,O)=(p_{1},\ldots,p_{k})$. The set $O\setminus \{p_{1},\ldots,p_{k}\}$ consists of a single element, which we call the residue and denote by $\rho(E,O)$.

Let us state the properties of the parking process which matter for our construction. In what follows, we call shift modulo $n$ the action of $\Z/n\Z$ on $\int{1}{n}^{k}$ and $\binom{\int{1}{n}}{k+1}$ determined componentwise and elementwise in the ovious way by the $n$-cycle $(1\ldots n)$.

\begin{lemma}\label{prop park} 1. The parking process is equivariant with respect to the shift modulo $n$, that is, $\Pi((1\ldots n)E,(1\ldots n)O)=(1\ldots n)\Pi(E,O)$ and $\rho((1\ldots n)E,(1\ldots n)O)=(1\ldots n)\rho(E,O)$\\
2. If $E'$ differs from $E$ by a permutation, then $\Pi(E',O)$ differs from $\Pi(E,O)$ by a permutation. In particular, $\rho(E',O)=\rho(E,O)$.\\
3. If $\rho(E,O)=1$, then for all $l\in \int{1}{k}$, one has $e_{l}<p_{l}$.
\end{lemma}

\begin{proof} The shift modulo $n$ is an automorphism of the set $\int{1}{n}$ endowed with the cyclic order determined by $(1\ldots n)$. The definition of the parking process by the equations \eqref{def jk} and \eqref{def other j} uses only this structure of cyclic order. Hence, the first assertion holds.

In order to check the second assertion, it suffices to check that the set $\{p_{1},\ldots,p_{k}\}$ is not modified by the permutation of two neighbours in the sequence $E$. This is a simple verification which we leave to the reader.

Let us assume that $\rho(E,O)=1$. Then, for all $l\in \int{1}{k}$, the integer $1$ belongs to $\{o_{1},\ldots,o_{k+1}\}\setminus\{p_{l+1},\ldots,p_{k}\}$. Hence, the integer $r \in \int{1}{n}$ such that $p_{l}=(1\ldots n)^{r} e_{l}$ satisfies $r\leq n+1-e_{l}$, actually even $r<n+1-e_{l}$ because $p_{l}\neq 1$, so that $e_{l}<p_{l}\leq n$.
\end{proof}

Let us now begin the construction of the surjection itself. Consider $A=(a_{1},\ldots,a_{k})\in \int{1}{n}^{k}$ and $B=\{b_{1},\ldots,b_{k+1}\} \subset \int{1}{n}$. Let us apply to $A$ and $B$ the shift modulo $n$ which ensures that the residue of the parking process applied to $A$ and $B$ is $1$. Thus, let us define $\tilde A=(1\ldots n)^{1-\rho(A,B)}A$ and $\tilde B=(1\ldots n)^{1-\rho(A,B)}B$.

Let $I=(i_{1}\leq \ldots \leq i_{k})$ be the non-decreasing reordering of $\tilde A$. Let $J=(j_{1},\ldots,j_{k})=\Pi(I,\tilde B)$ be the result of the parking process applied to $I$ and $\tilde B$.

\begin{lemma}\label{i<j} The inequalities $i_{1}<j_{1}, \ldots, i_{k}<j_{k}$ hold.
\end{lemma}

\begin{proof} By the first assertion of Lemma \ref{prop park}, we have $\rho(\tilde A,\tilde B)=1$. Since $I$ differs from $\tilde A$ by a permutation, the second assertion of the same lemma implies that $\rho(I,\tilde B)=1$. The third assertion of the same lemma concludes the proof.
\end{proof}

The main property of the construction so far is the following.

\begin{lemma}\label{sig} The chain $((i_{1}\, j_{1}),\ldots,(i_{k}\, j_{k}))$ belongs to $\Sigma_{n}^{*}(k)$.
\end{lemma}

\begin{proof} Since the sequence $(i_{1},\ldots,i_{k})$ is non-decreasing, it suffices to check that the second property of Proposition \ref{carac} is satisfied. Let us choose $l,m\in \int{1}{k}$ with $l<m$ and let us assume that $j_{l}>i_{m}$. We need to prove that $j_{l}>j_{m}$.

We have $j_{m}=\min\left(\int{i_{m}+1}{n}\cap (\{\tilde b_{1},\ldots,\tilde b_{k+1}\}\setminus \{j_{m+1},\ldots,j_{k}\})\right)$ and, since we are assuming that $j_{l}>i_{m}$, $j_{l}=\min\left(\int{i_{m}+1}{n}\cap (\{\tilde b_{1},\ldots,\tilde b_{k+1}\}\setminus \{j_{l+1},\ldots,j_{k}\})\right)$. Thus, $j_{l}$ is the minimum of a set which is contained in another set of which $j_{m}$ is the minimum. 
\end{proof}


In order to complete the construction, let us choose a permutation $\sigma\in \S_{k}$ such that $\sigma\cdot \tilde A=I$. There is in general more than one choice for $\sigma$, but two different choices belong to the same right coset of the stabilizer of $I$. Since the mapping $P$ preserves the stabilisers (see Proposition \ref{action}), the element 
\[\Gamma_{n,k}(A,B)=\sigma^{-1}\cdot ((i_{1}\, j_{1}),\ldots,(i_{k}\, j_{k}))\]
of $\Sigma_{n}(k)$ is well defined.

\begin{theorem} The mapping $\displaystyle \Gamma_{n,k}:\int{1}{n}^{k} \times \binom{\int{1}{n}}{k+1}  \to \Sigma_{n}(k)$
is a surjection whose fibres are the orbits of the shift modulo $n$. In particular, the preimage of each element of $\Sigma_{n}(k)$ contains $n$ elements and $\displaystyle |\Sigma_{n}(k)|=n^{k-1}\binom{n}{k+1}$.
\end{theorem}

\begin{proof} In order to prove that the mapping $\Gamma_{n,k}$ is surjective, let us construct a section of it.  

Let $\gamma=((i_{1}\, j_{1}),\ldots,(i_{k}\, j_{k}))$ be an element of $\Sigma_{n}(k)$. Let $\sigma\in \S_{k}$ be a permutation such that $\sigma\cdot \gamma=((a_{1}\, b_{1}),\ldots,(a_{k}\, b_{k}))$ satisfies $a_{1}\leq \ldots \leq a_{k}$. Set $b_{k+1}=1$. By Lemma \ref{sm}, the set $\{b_{1},\ldots,b_{k+1}\}$ contains $k+1$ elements. Proposition \ref{carac} implies that for all $l\in \int{1}{k}$, the set $\{a_{l}+1,\ldots,b_{l}\}\cap\{b_{1},\ldots,b_{l-1}\}$ is empty. Thus,
\[b_{l}=\min\left(\int{i_{l}+1}{n} \cap (\{b_{1},\ldots,b_{k+1}\}\setminus\{b_{l+1},\ldots,b_{k})\right),\]
so that $\Pi((a_{1},\ldots,a_{k}),\{b_{1},\ldots,b_{k+1}\})=(b_{1},\ldots,b_{k})$. Moreover, the residue of this parking process is $b_{k+1}=1$. It follows from the definition of $\Gamma_{n,k}$ that $\Gamma_{n,k}((i_{1},\ldots,i_{k}),\{b_{1},\ldots,b_{k+1}\})=\gamma$.

The definition of $\Gamma_{n,k}(A,B)$ shows that it is actually a function of $(1\ldots n)^{1-\rho(A,B)}A$ and $(1\ldots n)^{1-\rho(A,B)}B$. This observation and the first assertion of Lemma \ref{prop park} imply that $\Gamma_{n,k}$ is invariant under the shift modulo $n$.

Let us finally prove that each fibre of $\Gamma_{n,k}$ consists in one single orbit of the shift. Let $(A,B)\in \int{1}{n}^{k} \times \binom{\int{1}{n}}{k+1}$ be such that $\rho(A,B)=1$. Let $\sigma$ be a permutation which reorders $P(\Gamma_{n,k}(A,B))$ into a non-decreasing sequence and write $\sigma\cdot \Gamma_{n,k}(A,B)=((i_{1}\, j_{1}),\ldots,(i_{k}\, j_{k}))$. Then $A=\tilde A=P(\Gamma_{n,k}(A,B))$ and $B=\tilde B=\{1,j_{1},\ldots,j_{k}\}$. Hence, each fibre of $\Gamma_{n,k}$ contains a unique pair $(A,B)$ such that $\rho(A,B)=1$.

To be complete, one should conclude by observing that the action of $\Z/n\Z$ on $\int{1}{n}^{k} \times \binom{\int{1}{n}}{k+1}$ is free.
\end{proof}
\medskip

{\bf Acknowledgement.} It is a pleasure to thank Philippe Marchal who brought parking functions and their relation to the enumeration of minimal factorisations to my attention.

\bibliographystyle{plain}
\bibliography{geodesics}

\begin{thebibliography}{1}

\bibitem{Artin}
Emil Artin.
\newblock Theory of braids.
\newblock {\em Ann. of Math. (2)}, 48:101--126, 1947.

\bibitem{Biane}
Philippe Biane.
\newblock Minimal factorizations of a cycle and central multiplicative
  functions on the infinite symmetric group.
\newblock {\em J. Combin. Theory Ser. A}, 76(2):197--212, 1996.

\bibitem{LevyAIM}
Thierry L{\'e}vy.
\newblock Schur-{W}eyl duality and the heat kernel measure on the unitary
  group.
\newblock {\em Adv. Math.}, 218(2):537--575, 2008.

\bibitem{Ripoll}
Vivien Ripoll.
\newblock Orbites d'{H}urwitz des factorisations primitives d'un \'el\'ement de
  {C}oxeter.
\newblock {\em J. Algebra}, 323(5):1432--1453, 2010.

\bibitem{Stanley}
Richard~P. Stanley.
\newblock Parking functions and noncrossing partitions.
\newblock {\em Electron. J. Combin.}, 4(2):Research Paper 20, approx. 14 pp.
  (electronic), 1997.
\newblock The Wilf Festschrift (Philadelphia, PA, 1996).

\end{thebibliography}

\end{document}